\documentclass[11pt,twoside]{amsart}
\usepackage[latin1]{inputenc}
\usepackage[T1]{fontenc}
\usepackage{mathtools}
\usepackage{graphicx}
\usepackage{tikz}
\usetikzlibrary{chains}
\usepackage{tcolorbox}
\usepackage{mathabx}

\usepackage[latin1]{inputenc}
\usepackage{amsmath}
\usepackage{amsthm}
\usepackage{amssymb}

\usepackage[all]{xy}
\usepackage{hyperref}
\newtheorem{thm}{Theorem}[section]

\newtheorem{prop}[thm]{Proposition}

\newtheorem{lem}[thm]{Lemma}
\newtheorem{cor}[thm]{Corollary}
\newtheorem{conj}[thm]{Conjecture}

\theoremstyle{definition}
\newtheorem{definition}[thm]{Definition}
\newtheorem{remark}[thm]{Remark}
\newtheorem{ex}[thm]{Example}
\numberwithin{equation}{section}

\renewcommand{\to}{\xymatrix@1@=15pt{\ar[r]&}}
\renewcommand{\rightarrow}{\xymatrix@1@=15pt{\ar[r]&}}
\renewcommand{\leftarrow}{\xymatrix@1@=15pt{&\ar[l]}}
\renewcommand{\mapsto}{\xymatrix@1@=15pt{\ar@{|->}[r]&}}
\renewcommand{\twoheadrightarrow}{\xymatrix@1@=18pt{\ar@{->>}[r]&}}
\renewcommand{\hookrightarrow}{\xymatrix@1@=15pt{\ar@{^(->}[r]&}}
\newcommand{\congpf}{\xymatrix@L=0.6ex@1@=15pt{\ar[r]^-\sim&}}
\renewcommand{\cong}{\simeq}

\newcommand{\kn}{\mathcal N}
\newcommand{\kt}{\mathcal T}
\newcommand{\ko}{\mathcal O}
\newcommand{\QQ}{\mathbb Q}
\newcommand{\CC}{\mathbb C}
\newcommand{\ZZ}{\mathbb Z}
\newcommand{\RR}{\mathbb R}
\newcommand{\PP}{\mathbb P}
\DeclareMathOperator{\Ima}{Im}
\DeclareMathOperator{\Ker}{Ker}
\newcommand{\Pic}{\mathrm{Pic}}
\newcommand{\Gr}{\mathrm{Gr}}
\newcommand{\nilp}{\mathrm{nilp}}
\newcommand{\level}{\mathrm{level}}

\newcommand{\TBC}[1]{}

\usepackage{comment}
\excludecomment{comment}
\begin{document}

\title[]{Lagrangian fibrations}

\author[D.\ Huybrechts, M.\ Mauri, ]{ D.\ Huybrechts \& M.\ Mauri}

\address{Mathematisches Institut, Universit\"at Bonn, Endenicher Allee 60, 53115 Bonn, Germany \& Max Planck Institute for Mathematics, Vivatsgasse 7, 53111 Bonn, Germany}
\email{huybrech@math.uni-bonn.de \& mauri@mpim-bonn.mpg.de}

\begin{abstract} \noindent
We review the theory of Lagrangian fibrations of hyperk\"ahler manifolds
as initiated by Matsushita \cite{Matsushita:Fibrations,Matsushita:Addendum,Mat05:Matsushitahigher}. We also discuss more recent work of
Shen--Yin \cite{ShYi18:ShenYin} and Harder--Li--Shen--Yin \cite{HLSY19:ShenYin}.
Occasionally, we give alternative arguments and complement the discussion by additional observations.
 \vspace{-2mm}
\end{abstract}

\maketitle
{\let\thefootnote\relax\footnotetext{This review was prepared in the context of the
seminar organized by the ERC Synergy Grant HyperK, Grant agreement ID 854361.
The talk was delivered on June 4, 2021.}}
\marginpar{}

Assume $f\colon X\to B$  is a Lagrangian fibration of a compact hyperk\"ahler manifold $X$ of complex dimension $2n$, and $\pi\colon \mathcal{X} \to \Delta$ is a type III degeneration of compact hyperk\"{a}hler manifolds of complex dimension $2n$. Then the cohomology algebra of $\PP^n$ appears naturally in (at least) four different guises:
\smallskip

 (i) As the cohomology algebra of $(0,p)$ resp.\ $(p,0)$-forms (both independent of  $f$):$$H^*(\PP^n,\CC)\cong H^*(X,\ko_X)\text{ and } H^*(\PP^n,\CC)\cong H^0(X,\Omega^\ast_X).$$

(ii) As the cohomology of the base of the fibration:\footnote{Here and in (iii) and (iv), one expects isomorphisms of $\QQ$-algebras, but this seems not known.}
$$H^*(\PP^n,\CC)\cong H^*(B,\CC).$$

(iii) As the image of the restriction to the generic fiber $X_t$ of $f$:
$$H^*(\PP^n,\CC)\cong \Ima\left(H^*(X,\CC)\to H^*(X_t,\CC)\right).$$

(iv) As the cohomology of the dual complex $\mathcal{D}(\mathcal{X}_0)$ of the central fiber $\mathcal{X}_0$ of $\pi$:
$$H^*(\PP^n,\CC)\cong H^*(\mathcal{D}(\mathcal{X}_0), \CC).$$

In this survey we discuss these four situations and explain how they are related. We start by reviewing basic results on Lagrangian fibrations in Section \ref{sec:Basics},
discuss the topology of the base and the restriction to the fiber
in Section \ref{sec:Hbasefiber}, and then explain in Section \ref{sec:PW} how the various occurrences of $\PP^n$ are related, by sketching the proof of a key identity called P$=$W.
\smallskip

Throughout, $X$ denotes a compact hyperk\"ahler manifold of complex dimension $2n$.
A fibration of $X$ is a surjective morphism $f\colon X\twoheadrightarrow B$ with connected fibers onto a normal variety $B$ with $0<\dim(B)<2n$. 
A submanifold $T\subset X$ of dimension $n$ is Lagrangian if the restriction $\sigma|_T\in H^0(T,\Omega_T^2)$ of the holomorphic two-form $\sigma\in H^0(X,\Omega_X^2)$ is zero.

\section{Basics on Lagrangian fibrations}\label{sec:Basics}
We first discuss Lagrangian submanifolds and in particular Lagrangian tori. Then we study the cohomology and the singularities of the base $B$. 
Next we show that the fibers, smooth ones as well as singular ones, of any fibration are Lagrangian and conclude that fibrations of hyperk\"ahler manifolds over a smooth base are flat.

At the end, we mention further results and directions without proof:
Matshushita's description of the higher direct image sheaves $R^if_\ast\ko_X$,
Beauville's question whether Lagrangian tori are always Lagrangian fibers, 
smoothness of the base, etc.

\subsection{Lagrangian tori} We start with some general comments on Lagrangian manifolds and more specifically on Lagrangian tori.

\begin{prop}[Voisin]
Any Lagrangian submanifold $T\subset X$ of a hyperk\"ahler manifold $X$ is
projective. In particular, any Lagrangian torus is an abelian variety.
\end{prop}

\begin{proof} We follow the proof as presented in \cite{Ca06:CampanaIso}.
Since the restriction of any K\"ahler class on $X$ to $T$ is non-trivial, the restriction $H^2(X,\RR)\to H^2(T,\RR)$ is a non-trivial morphism of Hodge structures. On the other hand, as $T$ is Lagrangian, all classes in $H^{2,0}(X)\oplus H^{0,2}(X)$ have trivial restrictions. Hence, the image of $H^2(X,\RR)\to H^2(T,\RR)$ is contained in $H^{1,1}(T,\RR)$. More precisely, the images of $H^2(X,\RR)\to H^2(T,\RR)$ and of $H^{1,1}(X,\RR)\to H^{1,1}(T,\RR)$ coincide. Therefore, for any K\"ahler class $\omega\in H^{1,1}(X,\RR)$ there exists a rational class $\alpha\in H^2(X,\QQ)$ such that the $(1,1)$-class
$\alpha|_T$ comes arbitrarily close to the K\"ahler class $\omega|_T$. Thus, $\alpha|_T$ is a rational K\"ahler class and, hence, $T$ is projective.\end{proof}


\begin{remark}
The normal bundle of  a Lagrangian submanifold $T\subset X$ is isomorphic to the cotangent bundle of $T$, so $\kn_{T/X}\cong\Omega_T$. Hence, the $(1,1)$-part of the restriction
map $ H^2(X,\CC)\to H^2(T,\CC)$ can be identified with the natural map
$H^1(X,\kt_X)\to H^1(T,\kn_{T/X})$ that sends a first order deformation of $X$ to the obstruction to deform $T$ sideways with it, see \cite{Voi92:VoisinStabilityLagr}:
$$\xymatrix{H^{1,1}(X)\ar[d]^{\cong}\ar[r]&H^{1,1}(T)\ar[d]^{\cong}\\
H^1(X,\kt_X)\ar[r]& H^1(T,\kn_{T/X}).}$$
Clearly, as $T$ is Lagrangian, the map $(H^{2,0}\oplus H^{0,2})(X)\to H^2(T,\CC)$ is trivial, see \ the proof above. Since the restriction of a K\"ahler class is again K\"ahler, $H^{1,1}(X)\to H^{1,1}(T)$
is certainly not trivial. Thus, $T\subset X$ deforms with $X$ along a subset of codimension at least one. For smooth fibers of a Lagrangian fibration, so
eventually Section \ref{sec:More} for all Lagrangian tori, the rank of the restriction map  and hence the codimension of the image $\text{Def}(T\subset X)\to\text{Def}(X)$ is exactly one.\footnote{Is there an a priori reason why this is the case
for Lagrangian tori? It fails for general Lagrangian submanifolds; see \S\ref{sec:ex}.}
\end{remark}

\begin{prop}
Assume $T\subset X$ is a Lagrangian torus. Then the restrictions ${\rm c}_i(X)|_T\in H^{2i}(T,\RR)$
of the Chern classes ${\rm c}_i(X)\in H^{2i}(X,\RR)$  are trivial.
\end{prop}

\begin{proof}
The normal bundle sequence allows one to compute the restriction of the total
Chern class of $X$  to $T$. More precisely, ${\rm c}(\kt_X)|_T={\rm c}(\kt_T)\cdot{\rm c}({\mathcal N}_{T/X})$. To conclude, use ${\mathcal N}_{T/X}\cong\Omega_T$ and the fact that the tangent bundle of a torus is trivial.
\end{proof}

\begin{remark} (i) In the case when $T\subset X$ is the fiber of a Lagrangian fibration
$f\colon X\to B$, as it always is, see \ Section \ref{sec:More}, the restriction of the Beauville--Bogomolov--Fujiki form, thought of as a class  $\tilde q\in H^4(X,\QQ)$, is also trivial:
$$\tilde q|_T=0.$$ 
There does not seem to be a direct proof of this fact. However, using that the rank
of the restriction map $H^4(X,\QQ)\to H^4(T,\QQ)$ is one, see Theorem \ref{prop:CohBFibre}, it can
be shown as follows.
The classes $\tilde q$ and ${\rm c}_2$ in $H^4(X,\QQ)$ both have the distinguished property that the homogenous forms  $\int_X\tilde q\cdot\beta^{2n-2}$ and  $\int{\rm c}_2(X)\cdot\beta^{2n-2}$ defined on $H^2(X,\ZZ)$
are non-trivial scalar multiples  of $q(\beta)^{n-1}$ and, therefore, of each other.\footnote{The non-triviality of the scalar for ${\rm c}_2(X)$ follows from the fact that $\int_X{\rm c}_2(X)\cdot \omega^{2n-2}\ne0$ for any K\"ahler class $\omega$.}
If $[T]\in H^{2n}(X,\ZZ)$ is the class of a fiber $f^{-1}(t)$, then up to scaling $[T]=f^\ast\alpha^n$ for some $\alpha\in H^2(B,\QQ)$. Hence, for a K\"ahler class $\omega$ on $X$ we find (up to a non-trivial scalar factor)$$\int_T\tilde q|_T\cdot
\omega|_T^{n-2}=\int_X\tilde q\cdot f^\ast\alpha^n\cdot\omega^{n-2}=\int_X{\rm c}_2(X)
\cdot f^\ast\alpha^n\cdot\omega^{n-2}=\int_T{\rm c}_2(X)|_T\cdot\omega|_T^{n-2}=0.$$ 
Since $\omega|_T\ne0$ and $\Ima\left(H^\ast(X,\RR)\to H^*(T,\RR)\right)$ is generated by $\omega|_T$, this proves the claim.
\smallskip

(ii) For other types of Lagrangian submanifolds, the restrictions of the Chern classes of $X$ are not trivial. For example, for a Lagrangian plane $\PP^2\subset X$ one easily computes $\int_{\PP^2}{\rm c}_2(X)|_{\PP^2}=15$.
\end{remark}

As remarked before, the normal bundle of a Lagrangian torus is trivial.
The next observation  can be seen as a converse, it applies in particular to the smooth fibers of any fibration $f\colon X\to B$.

\begin{lem}\label{rem:Laggivestorus}
Assume $T\subset X$ is Lagrangian submanifold with trivial normal bundle. Then
$T$ is a complex torus and, therefore, an abelian variety.
\end{lem}

\begin{proof}
Since $T$ is Lagrangian, the tangent bundle  $\kt_T\cong{\mathcal N}^\ast_{T/X}$ is trivial. Using the Albanese morphism, one easily proves that any compact K\"ahler manifold with trivial tangent bundle is a complex torus.
\end{proof}

\subsection{The base of a fibration} We pass on to (Lagrangian) fibrations.

\begin{prop}[Matsushita]\label{prop:Matsub2}
Assume $f\colon X\twoheadrightarrow B$ is a fibration with
$B$ smooth. Then $B$ is a simply connected, smooth projective variety of dimension $n$ satisfying
$H^{p,0}(B)=H^{0,p}(B)=0$ for all $p>0$ and $H^2(B,\QQ)\cong\QQ$.
In particular, $${\rm Pic}(B)\cong H^2(B,\ZZ)\cong\ZZ.$$
\end{prop}

\begin{proof} The smoothness of $B$ implies that the pull-back $f^\ast\colon H^*(B,\QQ)\to H^*(X,\QQ)$ is injective; see Remark \ref{rmk:smoothness}. Next, as $\alpha^{2n}=0$ for any class $\alpha\in H^2(B,\RR)$, we have $(f^\ast\alpha)^{2n}=0$
and, therefore, $q(f^\ast\alpha)=0$. By \cite{Bo96:BogomolovSH,Ver96:VerbitskySH}, this implies $(f^\ast\alpha)^{n+1}=0$
and hence $\alpha^{n+1}=0$, which implies $\dim(B)\leq n$. On the other hand, again by \cite{Bo96:BogomolovSH,Ver96:VerbitskySH}, $(f^\ast\alpha)^n\ne0$ for every nonzero class $\alpha\in H^2(B,\RR)$
from which we deduce $n\leq \dim(B)$.

 If $\alpha\in H^{p,0}(B)$, then $f^\ast\alpha$ is a non-trivial multiple of some power of $\sigma$. 
Hence, $\alpha=0$ if $p$ is odd. If $p=2$, then $f^\ast\alpha=\lambda\cdot\sigma$ and,
hence, $f^\ast \alpha^n=\lambda^n\cdot\sigma^n$. Since $\sigma^n\ne0$ and $H^{2n,0}(B)=0$, one finds $\lambda=0$. A similar argument can be made to work for all even $p$
and an alternative argument is provided by Theorem \ref{prop:CohBFibre}.

Next we show $H^2(B,\QQ)\cong \QQ$. Using \cite{Bo96:BogomolovSH,Ver96:VerbitskySH}, we have $$S^n f^\ast H^2(B,\QQ)\subset S^nH^2(X,\QQ)\subset H^{2n}(X,\QQ).$$On the other hand, the image
of $S^nf^\ast H^2(B,\QQ)$ is contained in $f^\ast H^{2n}(B,\QQ)$ which is just one-dimensional.\footnote{The traditional proof goes as follows:
First one shows that for any non-trivial class $\alpha\in H^2(B,\RR)=H^{1,1}(B,\RR)$ 
and any K\"ahler class $\omega$ on $X$ one 
has $\int_X(f^\ast\alpha)\wedge\omega^{2n-1}\ne0$. Indeed, otherwise the Hodge index theorem
would imply $q(f^\ast \alpha)<0$ and, therefore, $(f^\ast \alpha)^{n+1}\ne0$, which
contradicts $\dim (B)=n$. As a consequence, observe that for any two non-trivial
classes $\alpha_1,\alpha_2\in H^2(B,\RR)$ there exists a linear combination
$\alpha\coloneqq\lambda_1\alpha_1+\lambda_2\alpha_2$ with $\int_X(f^*\alpha)\wedge\omega^{2n-1}=0$, which then implies $\alpha=0$, i.e.\ any two classes
$\alpha_1,\alpha_2\in H^2(B,\RR)$ are linearly dependent.}

Since $X$ is K\"ahler, so is $B$, see \cite{Va84:Varouchas}.
  Using $H^{2,0}(B)=H^{0,2}(B)=0$, we can conclude that there exists a rational
K\"ahler class on $B$. Hence, $B$ is projective. According to \cite[Prop.\ 2.10.2]{Ko95:KollarShafarevich}, the natural map
$\pi_1(X)\to\pi_1(B)$ is surjective and,  therefore, $B$ is simply connected, as $X$ is.\footnote{By Lemma \ref{lem:Fano} below, $B$ is a Fano manifold, which provides an alternative argument of the simply connectedness of $B$.} Then, by the universal coefficient theorem,  $H^2(B,\ZZ)$ is torsion-free, i.e.\ $H^2(B,\ZZ)\cong\ZZ$. Since $H^{1,0}(B)=H^{2,0}(B)=0$, the exponential sequence
gives  ${\rm Pic}(B)\congpf H^2(B,\ZZ)$.
\end{proof}

\begin{remark}\label{rem:cohBase}
In fact, as we shall see, $H^{p,q}(B)=0$ for all $p\ne q$ and $H^{p,p}(B)\cong H^{p,p}(\PP^n)$, i.e.\ there is an 
isomorphism of rational Hodge structures $$H^*(B,\QQ)\cong H^*(\PP^n,\QQ).$$
There are two proofs of this fact, both eventually relying on the isomorphism
$H^\ast(X,\ko_X)\cong H^\ast(\PP^n,\CC)$. It seems that unlike
$H^2(B,\QQ)\cong\QQ$, which was proved above by exploiting the  structure
of the subring of $SH^2(X,\QQ)\subset H^\ast(X,\QQ)$, the proof of the identities $H^{k}(B, \QQ)\simeq H^{k}(\PP^n, \QQ)$ for $k>2$ uses deeper information about the hyperk\"ahler structure.
\smallskip

(i) The first proof for $B$ smooth and $X$ projective was given by Matsushita \cite{Mat05:Matsushitahigher}, as a consequence of the isomorphisms $R^if_\ast\ko_X\cong\Omega_B^i$, see Section \ref{sec:Matdirect}.
Combining this isomorphism with the splitting $Rf_*\ko_X\cong\bigoplus R^if_*\ko_X[-i]$, see \cite{Ko86:KollarHigherII}, one finds $$H^k(X,\ko_X)\cong H^k(B,Rf_\ast\ko_X)\cong\bigoplus H^{k-i}(B,R^if_\ast\ko_X)\cong \bigoplus H^{k-i}(B,\Omega^i_B),$$ which proves the claim.\footnote{By evoking results due to Saito \cite{Saito}, it should be possible to avoid the projectivity assumption in  \cite{Ko86:KollarHigherII}.}
\smallskip

(ii) Another one, which also works for singular $B$ and non-projective $X$, was given in \cite{ShYi18:ShenYin} and roughly relies on the fact that $H^\ast(B,\CC)$ can be deformed into $H^\ast(X,\ko_X)$, see Section \ref{sec:geomreal}.
\end{remark}

\begin{lem}[Markushevich, Matsushita]\label{lem:Fano}
Under the above assumptions, $B$ is a Fano variety, i.e.\ $\omega_B^\ast$ is ample.
\end{lem}

\begin{proof}
Since $B$ is dominated by $X$, we have ${\rm kod}(B)\leq 0$ by the known case of the Iitaka conjecture; see \cite[Cor.\ 1.2]{Kawamata1985}. 
Hence, $\omega_B\cong\ko_B$ or $\omega_B^\ast$ is ample. 
However, the first case is excluded by $H^{n,0}(B)=0$. 

In \cite[Prop.\ 24.8]{GHJ03:CY_manifolds} the assertion is deduced from the fact that $X$ admits a K\"ahler--Einstein metric. The case $\omega_B\cong\ko_B$ is excluded, because
it would imply $H^{n,0}(B)\ne0$, which was excluded above.
\end{proof}

\begin{remark}
 It turns out that as soon as the base $B$ is smooth, then $B\cong\PP^n$.
 This result is due to Hwang \cite{Hwang:fibrationsbase} and its proof relies on the theory
of minimal rational tangents. The results by Matsushita and more recently
 by Shen and Yin, see Remark \ref{rem:cohBase} and Section \ref{sec:Hbasefiber},
 can be seen as strong evidence for the result.
 In dimension two, the result is immediate: Any smooth projective surface $B$ with $\omega_B^\ast$ ample and $H^2(B,\QQ)\cong\QQ$ is isomorphic to $\PP^2$.

It is tempting to try to find a more direct argument in higher dimension, but all attempts have failed so far. For example, according to Hirzebruch--Kodaira \cite{Hiko57:HirzebruchKodaira} it suffices to show that $H^*(B,\ZZ)\cong H^*(\PP^n,\ZZ)$ such that the first Chern class of a line bundle $L$ corresponding to a generator of $H^2(B,\ZZ)$ 
satisfies $h^0(B,L^k)=h^0(\PP^n,\ko(k))$, see \cite{Li16:LiProjectiveSpace} for a survey of further results in this direction. 

Alternatively, by Kobayashi--Ochai \cite{KoOc:KobayashiOchiai}, it is enough to show that $\omega_B$ is divisible by $n+1$, i.e.\ the Fano manifold $B$ has index $n+1$. As a first step, one could try to show that $f^\ast\omega_B$ is divisible by $n+1$.
\end{remark}

\subsection{Singularities of the base}\label{sec:sing}
It is generally expected that the base manifold $B$ is smooth, but at the moment this is only
known for $n\leq2$, see \ \cite{Ou19:OuLagrangianfibrations,BoKu18:BogomolovKurnosov,HuXu21:HuybrechtsXu}. The expectation is corroborated by the following computations of invariants of the singularities of $B$. 

Denote by $I\!H^*(B, \QQ)$ the intersection cohomology of the complex variety $B$ with middle perversity and rational coefficients. It is the hypercohomology of the intersection cohomology complex $\mathcal{IC}_B$, i.e.\ $I\!H^*(B, \QQ) = \mathbb{H}^*(B, \mathcal{IC}_B)$. In particular, if $B$ is smooth or has quotient singularities,
see \ \cite[Prop.\ 3]{GottscheSoergel993}, then $I\!H^*(B, \QQ) = H^*(B, \QQ)$.

\begin{prop}\label{prop:Bsing}
Assume $f\colon X \to B$ is a fibration over the complex variety $B$. 
\begin{enumerate}
    \item[(i)] $B$ is $\QQ$-factorial,\footnote{Are the singularities of $B$ actually factorial?} both in the Zariski and in the analytic topology.
    \item[(ii)] The intersection cohomology complex $\mathcal{IC}_B$ of $B$ is quasi-isomorphic to the constant sheaf $\QQ_B$. In particular, $I\!H^*(B, \QQ)=H^*(B, \QQ)$.
    \item[(iii)] \emph{(Matsushita)} $B$ has log terminal singularities.
\end{enumerate}
\end{prop}
\begin{proof} For (i) and (ii) one only needs that $f\colon X\to B$ is a connected and equidimensional morphism from a smooth variety $X$, while in the proof of (iii) one also
needs $\omega_X$ trivial.

For any $t \in B$, choose a chart $\varphi\colon U_x \subset X \to \CC^{2n}$, centered at $x$, and the analytic subset $S \coloneqq \varphi^{-1}(\Lambda)$, where $\Lambda \subseteq \CC^{2n}$ is an $n$-dimensional affine subspace intersecting the fiber $\varphi(f^{-1}(t))$ transversely. Since $f$ is equidimensional, 
the restriction $f|_S\colon S \to B$ is finite over an analytic neighbourhood $U$ of $t$. Therefore, $U$ is $\QQ$-factorial by \cite[Lem.\ 5.16]{KollarMori}. 

Denote $S^{\circ}\coloneqq S \cap f^{-1}(U)$. By the decomposition theorem \cite{BBD}.\footnote{Alternatively, note that the trace map $R(f|_{S^{\circ}})_* \QQ_{S^{\circ}} \to \mathcal{IC}_{U}$ splits the natural morphism $\mathcal{IC}_{U} \to R(f|_{S^{\circ}})_* \QQ_{S^{\circ}}$.} $\mathcal{IC}_{U}$ is a direct summand of $R(f|_{S^{\circ}})_* \QQ_{S^{\circ}}$. Taking stalks at $t$, we have
\[\mathcal{H}^0(\mathcal{IC}_B)_t \simeq \QQ_{B, t} \qquad \mathcal{H}^i(\mathcal{IC}_{U})_t \subseteq \mathcal{H}^i(R(f|_{S^{\circ}})_* \QQ_{S^{\circ}})_t=0,\]
because of the finiteness of $f|_{S^{\circ}}$.
Thus, the natural map $\QQ_B \to \mathcal{IC}_B$ is a quasi-isomorphism in the constructible derived category $D^b_c(B)$ with rational coefficients. 

By the canonical bundle formula, there exists a $\QQ$-divisor $\Delta \subset B$ such that the pair $(B, \Delta)$ is log terminal; see \cite[Thm.\ 8.3.7.(4)]{Kollar07} and \cite[Thm.\ 2]{Nakayama1988}. 
 By the $\QQ$-factoriality, $B$ has log terminal singularities too. 
\end{proof}
\begin{remark}[Quotient singularities]
The finiteness of the restriction $f|_S\colon S \to B$ over $b$ suggests that $B$ should have at worst quotient singularities. This would follow from the following conjecture.
\begin{conj}\label{Kconj}\emph{\cite[ \S 2.24]{Kollar2007:resolution}} 
Let $f\colon X \to Y$ be a finite and dominant morphism
from a smooth variety $X$ onto a normal variety $Y$. Then $Y$ has quotient singularities.
\end{conj}
This is known for $n=2$ by \cite[Lem.\ 2.6]{Brieskorn67}, but it is open in higher dimension. One of the main issue is that $f$ itself need not be a quotient
map, not even locally.  
\end{remark}
\begin{cor}\label{cor:injpullback}
The pullback $f^*\colon H^*(B, \QQ) \to H^*(X, \QQ)$ is injective.
\end{cor}
\begin{proof}
By Proposition \ref{prop:Bsing} this follows from the inclusion $I\!H^*(B, \QQ)\, \hookrightarrow H^*(X, \QQ)$ coming from the decomposition theorem. 
\end{proof}
\begin{remark}\label{rmk:smoothness}
Let $f\colon M \to N$ be a surjective holomorphic map between compact complex manifolds, with $M$ K\"{a}hler. By \cite[Lem.\ 7.28]{Voisin07I}, 
the pullback $f^*\colon H^*(N, \QQ) \to H^*(M, \QQ)$ is injective. However, this may fail if $N$ is singular, e.g.\ if $f$ is a normalization of a nodal cubic, even if $N$ has $\QQ$-factorial log terminal singularities, see for instance \cite[Thm.\ 5.11]{Mauri2021}.
\end{remark} 
\begin{remark}
Assume that $B$ is projective. By Corollary \ref{cor:injpullback}, the smoothness of $B$ can be dropped from the assumptions of Proposition \ref{prop:Matsub2} and Lemma \ref{lem:Fano}, see also \cite{Matsushita:Fibrations}.
\end{remark}

\subsection{The fibers of a fibration} Next we present Matsushita's result that any fibration 
of a compact hyperk\"ahler manifold is  a Lagrangian fibration.

\begin{lem}[Matsushita]
Assume $f\colon X\to B$ is a fibration. 
Then every smooth fiber $T\coloneqq X_t\subset X$ is a Lagrangian torus and in fact an abelian variety.
\end{lem}

\begin{proof} Comparing the coefficients of $x^{n-2}y^n$ in the polynomial (in $x$ and $y$) the equation
$$q(\sigma+\bar\sigma+x\cdot\omega+y\cdot f^\ast\alpha)^n=c_X\cdot\int_X(\sigma+\bar\sigma+x\cdot \omega+y\cdot f^\ast\alpha)^{2n}$$ shows $\int_X(\sigma\wedge\bar\sigma)\wedge\omega^{n-2}\wedge f^\ast(\alpha^n)=0$ for all $\omega\in H^2(X,\RR)$ and all $\alpha\in H^2(B,\RR)$.
Since $[T]=f^\ast(\alpha^n)$ for some class $\alpha$, this implies
$\int_F(\sigma\wedge\bar\sigma)|_T\wedge \omega^{n-2}|_T=0$, which for a K\"ahler class $\omega$
and using that $\sigma\wedge\bar\sigma$ is semi-positive implies $\sigma|_T=0$.
Then conclude by Lemma \ref{rem:Laggivestorus}.
\end{proof}

\begin{lem}[Matsushita]\label{lem:equidimensional}
The symplectic form $\sigma\in H^{2,0}(X)$ is trivial when restricted to any subvariety $T\subset X$ contracted to a point $t$ under $f$.
In particular, all fibers of $f$ are of dimension $n$, i.e.\ $f$ is equidimensional, and if $B$ is smooth, $f$ is flat.
\end{lem}

\begin{proof}
A theorem due to Koll\'{a}r \cite[Thm.\ 2.1]{Ko86:KollarHigherI} and Saito \cite[Thm.\ 2.3, Rem.\ 2.9.]{Saito} 
says that $R^2f_*\omega_X$ is torsion free. Since in our case $\omega_X \simeq \mathcal{O}_X$, this shows that $R^2f_*\mathcal{O}_X$ is torsion free. Let $\bar{\sigma} \in H^2(X, \mathcal{O}_X)$ be the conjugate of the symplectic form, and $\rho$ be its image in $ H^0(B, R^2f_*\mathcal{O}_X)$. Since the general fiber is Lagrangian, $\rho$ must be torsion and hence zero. If $\widetilde{T} \to T$ is a resolution of $T$, then the image of $\bar{\sigma}$ in $H^2(\widetilde{T}, \mathcal{O}_{\widetilde{T}})$ is contained in the image of
\[
R^2f_*\mathcal{O}_X \otimes k(t) \to H^2(T, \mathcal{O}_T)\to H^2(\widetilde{T}, \mathcal{O}_{\widetilde{T}})
\]
and hence trivial. This implies that the image of $\sigma$ in $H^0(\widetilde{T}, \Omega^2_{\widetilde{T}})$ is trivial, i.e.\ $\sigma|_T=0$. By semi-continuity of the dimension of the fibers, $\dim T \geq n$, and so $T$ is Lagrangian.

The flatness follows from the smoothness of $X$ and $B$, see \cite[Exer.\ III.10.9]{Hartshorne}.
\end{proof}

\begin{remark}
Note that the conclusion that $f$ is flat really needs the base to be smooth.
In fact, by miracle flatness, $f$ is flat if and only if $B$ is smooth.
\end{remark}

\subsection{Further results}
We summarize a few further results without proof.
\subsubsection{Higher direct images}\label{sec:Matdirect}
The first one is the main result of \cite{Mat05:Matsushitahigher}.
\begin{thm}[Matsushita]\label{thm:Matsushitabase}
Assume $f\colon X\to B$ is a fibration of a projective\footnote{Again, the projectivity assumption can presumably be dropped by applying results of Saito.}
hyperk\"ahler manifold over a smooth base. Then
$$R^if_*\ko_X\cong \Omega_B^i.$$
\end{thm}

On the open subset $B^{\circ} \subset B$ over which $f^{\circ}\coloneqq f|_{f^{-1}(B^{\circ})} \colon X^{\circ} \to B^{\circ}$ is smooth, the result can be obtained by dualising the isomorphism
\[
 f^{\circ}_*\Omega^1_{X^{\circ}/B^{\circ}} \simeq T_{B^{\circ}},
\]
which holds because the smooth fibers of $f$ are Lagrangian. A relative polarization is used to show that $R^1f^{\circ}_* \mathcal{O}_{X^{\circ}}$ and $f^{\circ}_*\Omega^1_{X^{\circ}/B^{\circ}}$ are dual to each other. To extend the result from $B^{\circ}$ to the whole $B$, Theorem \ref{thm:Matsushitabase} uses a result of Koll\'ar \cite[Thm.\ 2.1]{Ko86:KollarHigherI} saying that $R^if_\ast\omega_X$ are torsion free, which for $X$ hyperk\"ahler translates into $R^if_\ast\ko_X$ being torsion free.


As mentioned in Remark \ref{rem:cohBase}, the theorem implies $H^\ast(B,\QQ)\cong H^\ast(\PP^n,\QQ)$.
 
\subsubsection{Lagrangian tori are Lagrangian fibers}\label{sec:More}
In \cite{Be11:BeauvilleProblemlist} Beauville asked whether every Lagrangian torus $T\subset X$ is the fiber of a Lagrangian fibration $X\to B$. The question has been
answered affirmatively:  
\smallskip

(i) Greb--Lehn--Rollenske in \cite{GLR:Lagrangian1} first dealt with the case of non-projective $X$ and later  showed in \cite{GLR:Lagrangian2} the existence of an almost\footnote{A meromorphic map $f \colon X \dashrightarrow B$ is almost holomorphic if there exists a Zariski-open subset $U \subset B$ such that $f|_{f^{-1}(U)}\colon f^{-1}(U) \to U$ is holomorphic and proper.} holomorphic Lagrangian fibration in dimension four.
\smallskip

(ii) A different approach to the existence of an almost holomorphic Lagrangian fibration with $T$ as a fiber was provided by Amerik--Campana \cite{AmCa13:AmerikCampanaOnfamilies}. The four-dimensional case had been discussed
before by Amerik \cite{Am12:AmerikBeauvillequestion}.
\smallskip

(iii) Hwang--Weiss  \cite{HW:Lagrangian} deal with the projective case
and proved the existence of an almost  Lagrangian fibration with fiber $T$. Combined with techniques of \cite{GLR:Lagrangian1} 
this resulted in a complete answer.

\section{Cohomology of the base and cohomology of the fiber}\label{sec:Hbasefiber}
The aim of this section is to prove the following result.

\begin{thm}\label{prop:CohBFibre}
Assume $X\to B$ is a fibration 
and let $X_t$ be a smooth fiber.
Then
$$ H^\ast(\PP^n,\QQ)\cong H^\ast(B,\QQ)\text{ and } H^\ast(\PP^n,\QQ)\cong\Ima\left(H^\ast(X,\QQ)\to H^\ast(X_t,\QQ)\right).$$
\end{thm}

The first isomorphism for $X$ projective and $B$ smooth is originally due to Matsushita \cite{Mat05:Matsushitahigher}, see Remark \ref{rem:cohBase}.
The proof we give here is a version of the one by Shen and Yin \cite{ShYi18:ShenYin} that works without assuming  $X$ projective.  Note also that we do not assume that the base $B$ is smooth. 

The second isomorphism in degree two is essentially due to Oguiso \cite{Og09:Oguisoabelian}, relying on results of Voisin \cite{Voi92:VoisinStabilityLagr}.
The paper by Shen and Yin \cite{ShYi18:ShenYin} contains two proofs of the general result,
one using the ${\mathfrak s}{\mathfrak l}_2$-representation theory of the perverse filtration and another one, due to Voisin,
relying on classical Hodge theory.

The proof we shall give avoids the perverse filtration as well as the various ${\mathfrak s}{\mathfrak l}_2\times{\mathfrak s}{\mathfrak l}_2$-actions central for the arguments in \cite{ShYi18:ShenYin}.
The discussion below also proves the second result in \cite[Thm.\ 0.2]{ShYi18:ShenYin}, namely the equality
$$ {^{\mathfrak{p}}}h^{i,j}(X)=h^{i,j}(X)$$
between the classical and perverse Hodge numbers, see Section \ref{sec:perversefiltra}.
How it fits into the setting of  P$=$W  is explained in Section \ref{sec:PW}.

\subsection{Algebraic preparations}\label{sec:algpreps}
To stress the purely algebraic nature of what follows we shall use the shorthand $H^\ast\coloneqq H^\ast(X,\CC)$ and consider it as a graded $\CC$-algebra. 

Consider a non-trivial, isotropic  element $\beta$ of degree two, i.e.\ $0\ne\beta\in H^2$ with $q(\beta)=0$. Then, according
to Verbitsky and Bogomolov \cite{Bo96:BogomolovSH,Ver96:VerbitskySH}, one has
$$\beta^n\ne0\text{  and }\beta^{n+1}=0.$$ In particular, multiplication by
$\beta$ defines on $H^\ast$ the structure of a graded  $\CC[x]/(x^{n+1})$-algebra
with  $x$ of degree two.

All that is needed in the geometric applications is then put into the following statement.

\begin{prop}\label{prop:C[x]/x^n+1} For every two non-zero, isotropic elements $\beta,\beta'\in H^2$, the induced graded $\CC[x]/(x^{n+1})$-algebra structures
on $H^\ast$ are isomorphic.
\end{prop}

\begin{proof} 

Consider the complex algebraic group of
automorphisms ${\rm Aut}(H^\ast)$ of the graded $\CC$-algebra $H^\ast$ and
its image $G$ under ${\rm Aut}(H^\ast)\to {\rm Gl}(H^2)$. Clearly, the assertion
holds if $\beta,\beta'\in H^2$ are contained in the same $G$-orbit.
As any two non-zero isotropic classes $\beta,\beta'$ are contained in the
same orbit of the complex special orthogonal group ${\rm SO}(H^2,q)$, it suffices to show
that ${\rm SO}(H^2,q)\subset G$. This follows from \cite[Prop.\ 3.4]{SoldatenkovHodgeHK}, up to taking complex coefficients in loc.\ cit. 

\end{proof}

\begin{remark}
The arguments can be adapted to prove the following statement: Assume $\beta,\beta'\in H^2$ satisfy $q(\beta)=q(\beta')\ne0$. Then the induced graded $\CC[x]/(x^{2n+1})$-algebra structures on $H^\ast$, given by letting $x$ act by multiplication with $\beta$ resp.\ $\beta'$, are isomorphic.
\end{remark}

For $0\ne\beta\in H^2$ with $q(\beta)=0$ and $d\leq n$ we let
$$H^d_{\beta\text{-pr}}\coloneqq{\rm Ker}\left(\!\xymatrix{{\beta^{n-d+1}}\colon H^d\ar[r]&H^{2n-d+2}}\!\right),$$ which is called the space of $\beta$-primitive forms.
Note, however, that $\beta$ does not satisfy the Hard Lefschetz theorem; otherwise we would have defined primitive classes in $H^d$ as elements in the kernel of $\beta^{2n-d+1}$. 

We will also need the two spaces
\begin{equation}\label{eq:P0Q*}
P_0H^d\coloneqq \Ima\left(\!\xymatrix{\bigoplus_{d-2i\leq n}\beta^i \cdot H^{d-2i}_{\beta\text{-pr}}\ar[r]& H^d}\!\right)\text{ and }\bar P_{0}H^d\coloneqq H^d/\, {\rm Ker}\left(\!\xymatrix{\beta^n\colon H^d\ar[r]&H^{d+2n}}\!\right).
\end{equation}
It turns out that the map in the definition of $P_0$ is injective, but this is not needed for the argument. Note that $P_0H^d\subset {\rm Ker}(\beta^n)\subset H^d$ for all $d>0$.

\begin{cor}\label{cor:Pindependentalpha}
The dimensions of the spaces $P_0H^d$ and $\bar P_0H^d$ are independent of the choice of the non-trivial, isotropic class $\beta\in H^2$.\qed
\end{cor}

\subsection{Geometric realizations}\label{sec:geomreal}
Let us begin by looking at the obvious choice for $\beta$ provided by the symplectic form $\sigma\in H^0(X,\Omega_X^2)\subset H^2(X,\CC)$.

\begin{lem} For $\beta=\sigma$ one has
$$P_0H^d=H^0(X,\Omega_X^d)\subset H^d(X,\CC)\text{ and } P_0H^\ast\cong H^*(\PP^n,\CC)$$
and
$$\bar P_0 H^d\cong H^d(X,\ko_X)\text{ and } \bar P_{0} H^\ast\cong H^*(\PP^n,\CC).$$
\end{lem}

\begin{proof}
Concerning the first equality, one inclusion is obvious: Since $H^0(X,\ko_X)=H^0(X,\CC)=H^0_{\sigma\text{-pr}}$,
we have $H^0(X,\Omega_X^{d})=\CC\cdot\sigma^{d/2}\subset P_0H^{d}$ for $d$ even and $H^0(X,\Omega_X^d)=0$ for $d$ odd. For the other direction, use that $\sigma^{n-p}\colon \Omega_X^p\congpf \Omega_X^{2n-p}$, for $p\leq n$, is an isomorphism and that, therefore,
for $q>0$ the composition  \begin{equation}\label{eq:perverseHodgeHL}
\xymatrix@C=35pt{ H^{p,q}(X)\ar[r]^-{\sigma^{n-d+1}}& H^{2n-p-2q+2,q}(X)\ar[r]^-{\sigma^{q-1}}& H^{2n-p,q}(X)}
\end{equation} is injective. Hence, for $d\leq n$, we have
$\sigma^{n-d+1}$ is injective, i.e.\ $H^{p,q}(X)\cap H^d_{\sigma\text{-pr}}=0$ for $q>0$, which is enough to conclude.

For the second part observe that ${\rm Ker}(\sigma^n)\cap \bigoplus H^{p,q}(X)= 
\bigoplus_{p>0} H^{p,q}(X)$.
\end{proof}

As an immediate consequence of Corollary \ref{cor:Pindependentalpha} one then finds.

\begin{cor}\label{cor:P0P2nsame}
For any non-trivial, isotropic class $\beta\in H^2$ there exist isomorphisms
$$ P_0H^\ast\cong H^*(\PP^n,\CC)\text{ and }\bar P_0H^\ast\cong H^*(\PP^n,\CC)$$
  of graded vector spaces.\qed
\end{cor}

Next let us consider a Lagrangian fibration $f\colon X\to B$. 
We consider the
class $\beta\coloneqq f^\ast\alpha$, which is isotropic since $\alpha^{n+1}=0$ for dimension reasons.

\begin{lem} For $\beta=f^\ast\alpha$ there exists an inclusion
$$f^\ast H^\ast(B,\CC)\subset P_0H^\ast(X,\CC).$$ 
\end{lem}

\begin{proof}
The assertion follows from the Lefschetz decomposition $$H^d(B, \CC)=I\!H^d(B,\CC)=\bigoplus_i\alpha^i\cdot I\!H^{d-2i}(B,\CC)_{\rm pr}$$ on $B$, with respect to the unique ample class $\alpha\in H^2(B,\QQ)$, see \ \cite[Thm.\ 2.2.3.(c)]{deCataldoMigliorini2005}, and the observation that pull-back via $f$ maps $I\!H^{d-2i}(B,\CC)_{\rm pr}$
into $H^{d-2i}_{\beta\text{-pr}}$.
\end{proof}

Corollary \ref{cor:Pindependentalpha} then
immediately implies
$$ H^\ast(B,\CC)\cong P_0H^\ast\cong H^\ast(\PP^n,\CC),$$
see \ Remark \ref{rem:cohBase}, which proves the first part of Theorem \ref{prop:CohBFibre}.
\smallskip

We keep the isotropic class $\beta=f^\ast\alpha$ and 
observe that the natural inclusion 
\begin{equation}\label{eq:XtoXt}
{\rm Ker}\left(\!\!\xymatrix@C=15pt{H^d(X,\QQ)\ar[r]& H^d(X_t,\QQ)}\!\!\right)\subset {\rm Ker}\left(\!\!\xymatrix@C=25pt{[X_t]\colon H^d(X,\QQ)\ar[r]& H^{d+2n}(X,\QQ)}
\!\!\right).
\end{equation}
is actually an isomorphism.

\begin{lem}[Voisin]\label{lem:Voisinrest}
Let $\beta=f^\ast\alpha$ be as before and let $X_t\subset X$ be a smooth fiber of $f$. Then
$${\rm Ker}(\beta^n)\subset{\rm Ker}\left(\!\!\xymatrix@C=15pt{H^d(X,\QQ)\ar[r]& H^{d}(X_t,\QQ)}\!\!\right).$$
\end{lem}

\begin{proof}
The result is proved in \cite[App.\ B]{ShYi18:ShenYin}. The assertion is shown to be equivalent to the statement that the intersection pairing on the fiber is non-degenerate 
on the image of the restriction map, which in turn is deduced from Deligne's global 
invariant cycle theorem.\TBC{A priori it seems that one needs $X$ to be projective for this, but again  possibly one can use Saito's decomposition theorem (or rather some lemmas within it) instead. Or maybe Deligne's theorem just works fine. In Voisin's book it  is only needed for the mixed Hodge structure. Should be ok.}
\end{proof}

From the result one obtains a surjection $$\pi\colon \bar P_{0}H^\ast\twoheadrightarrow \Ima\left(H^\ast(X,\CC)\to H^\ast(X_t,\CC)\right).$$
Since $\bar P_{0}H^\ast \cong H^\ast(\PP^n,\CC)$ by Corollary \ref{cor:P0P2nsame}, its image in $H^\ast(X_t,\CC)$ is the subring generated by the restriction of a K\"ahler class.
Hence, $\pi$ is an isomorphism, which proves the second isomorphism in Theorem \ref{prop:CohBFibre}. However, it is easier to argue directly, as the equality holds in Lemma \ref{lem:Voisinrest} by \eqref{eq:XtoXt}. 

\subsection{}\label{sec:perversefiltra}
As in Section \ref{sec:algpreps}, we consider the abstract algebraic situation provided
by $H^\ast\coloneqq H^{\ast}(X,\CC)$ and the additional structure induced by the
choice of a non-zero isotropic class $\beta\in H^2$.
The two spaces $P_0H^d$ and $\bar P_0H^d$ defined there, both depending on $\beta$, are part 
of a filtration
$$P_0H^\ast\subset P_1H^\ast\subset\cdots \subset P_{2n-1}H^\ast\subset P_{2n}H^\ast=H^\ast,$$ where $P_0H^d$ is as defined before and $\bar P_0H^d=H^d/P_{d-1}H^d$.

In general, one defines 
\begin{equation}\label{eq:perverse}
P_k H^d\coloneqq\sum_{i\geq 0}
\beta^i \cdot {\rm Ker}\left(\beta^{n-(d-2i)+k+1}\colon H^{d-2i}\to H^{2n-d+2i+2k+2}\right).
\end{equation}
If we want to stress the dependence of $\beta$, we write $P_k^\beta H^d$.
The graded objects of this filtration
$$\Gr_i^PH^\ast\coloneqq P_iH^\ast/P_{i-1}H^\ast,$$
 in particular $\Gr_dH^d=\bar P_0H^d$,
are used to  define the \emph{Hodge numbers} of the filtration
as $${^{P}\!}h^{i,j}\coloneqq \dim \Gr_i^PH^{i+j}.$$

As a further consequence of Proposition \ref{prop:C[x]/x^n+1}, one
has

\begin{cor}\label{cor:C[x]/x^n+1}
The Hodge numbers ${^{P}\!}h^{i,j}$ of the filtration $P_iH^\ast$ are independent of the choice of the isotropic class $\beta\in H^2$.\qed
\end{cor}

Let us quickly apply this to two geometric examples. 
\smallskip

(i) First, consider $\beta=\bar\sigma\in H^2(X,\ko_X)\cong H^{0,2}(X)\subset H^2(X,\CC)$, the anti-holomorphic symplectic form. Then the filtration gives back the Hodge filtration, i.e.\
$$P_k^{\bar\sigma}H^d=\bigoplus_{p\leq k}H^{p,d-p}(X).$$
To see this, one needs to use the Lefschetz decomposition with respect to $\bar\sigma$:
$$H^q(X,\Omega_X^p)=\bigoplus_{q-\ell\geq (q-n)^+}{\bar\sigma}^{q-\ell}\cdot H^{2\ell-q}(X,\Omega_X^p)_{\bar\sigma\text{-pr}}.$$
Note that from this example one can deduce that indeed for any choice of $\beta$ one
has $P_k^\beta H^d=0$ for $k<0$ and $P_k^\beta H^d=H^d$ for $k\geq d$.

(ii) For the second example consider a Lagrangian fibration $f\colon X\to B$ and
let $\beta$ be the pull--back of an ample class $\alpha\in H^2(B,\QQ)$. The induced filtration
is called the \emph{perverse filtration}\footnote{The classical definition of the perverse filtration for the constructible complex $Rf_*\QQ_X$ due to \cite{BBD} or \cite[Def.\ 4.2.1]{deCataldoMigliorini2005} coincides with the present one; see \cite[Prop.\ 5.2.4.(39)]{deCataldoMigliorini2005}.} and the Hodge numbers are
denoted by ${^{\mathfrak{p}}}h^{i,j}(X)$.
\smallskip

Then \cite[Thm.\ 0.2]{ShYi18:ShenYin} becomes the following immediate consequence
of Proposition \ref{prop:C[x]/x^n+1} or Corollary \ref{cor:C[x]/x^n+1}.

\begin{cor}[Shen--Yin]\label{cor:ShenYin}
For any Lagrangian fibration $f\colon X\to B$ the Hodge numbers
of the perverse filtration equal the classical Hodge numbers:
$$ {^{\mathfrak{p}}}h^{i,j}(X)=h^{i,j}(X).$$
\end{cor}

\section{P$=$W }\label{sec:PW}
P$=$W for compact hyperk\"{a}hler manifolds asserts that the perverse filtration associated with a Lagrangian fibration can be realised as the weight filtration of a limit mixed Hodge structure of a degeneration of compact hyperk\"{a}her manifolds. It boils down to the observation that the cup product by a semiample not big class and a logarithmic monodromy operator define nilpotent endomorphisms in cohomology which are not equal, but up to renumbering induce the same filtration. 
 Inspired by P$=$W, we provide some geometric explanation or conjecture concerning the appearance of the cohomology of $\PP^n$ in the introduction and in
 Theorem \ref{prop:CohBFibre}.

\subsection{The weight filtration of a nilpotent operator}

\begin{definition}\label{defn:weightfiltration}
Given a nilpotent endomorphism $N$ of a finite dimensional vector space $H^\ast$ of index $l$, i.e.\ $N^{l}\neq 0$ and $N^{l+1}= 0$, the \emph{weight filtration of $N$ centered at $l$} is the unique increasing filtration
$$W_0H^\ast\subset W_1H^\ast\subset\cdots \subset W_{2l-1}H^\ast\subset W_{2l}H^\ast=H^\ast,$$
 with the property that (1) $N W_k \subseteq W_{k-2}$, and denoting again by $N$ the induced endomorphism on graded pieces, (2) $N^k \colon \Gr^W_{l+k} H^* \simeq \Gr^W_{l-k}H^*$ for every $k\geq 0$, see \cite[\S 1.6]{Deligne80}. 
\end{definition}

 The weight filtration of $N$ on $H^*$ can be constructed inductively as follows: first let $W_0\coloneqq \mathrm{Im} N^{l}$, and $W_{2l-1}\coloneqq \ker N^{l}$. We can replace $H^*$ with $W_{2l-1}/W_0$, on which $N$ is still well-defined and $N^l=0$. Then define
\begin{align*}
W_1& \coloneqq \text{inverse image in }W_{2l-1}\text{ of }\Ima N^{l-1}\text{ in }W_{2l-1}/W_0,\\
W_{2l-2}& \coloneqq \text{inverse image in }W_{2l-1}\text{ of }\Ker N^{l-1}\text{ in }W_{2l-1}/W_0.
\end{align*}
Continuing inductively, we obtain the unique (!) filtration on $H^*$ satisfying (1) and (2).

By the Jacobson--Morozov theorem, the nilpotent operator $N$ can be extended to an $\mathfrak{sl}_2$-triple with Cartan subalgebra generated by an element $H^{N}$ which is unique up to scaling. By the representation theory of $\mathfrak{sl}_2$-triples, there exists a decomposition
\[H^*=\bigoplus^{l}_{\lambda=-l} H^*_{\lambda},\]
called the \emph{weight decomposition}, with the property that $H^N(v)=\lambda v$ for all $v \in H^*_{\lambda}$. In particular, the decomposition splits the weight filtration of $N$
\[W_kH^*=\bigoplus^{l}_{\lambda=l-k} H^*_{\lambda}.\]

 Let us apply this to some geometric examples. 
\smallskip

(i) Any cohomology class $\omega \in H^2(X, \CC)$ defines a nilpotent operator $L_{\omega}$ on $H^*\coloneqq H^*(X, \CC)$ by cup product. If $\omega$ is K\"{a}hler, then the Hard Lefschetz theorem implies that the weight filtration of $L_{\omega}$ on $H^*$ centered at $2n$ is 
\[W^{\omega}_k H^*= \bigoplus_{i\geq 4n-k} H^{i}(X, \CC).\footnote{The equality actually holds for any K\"{a}hler manifold, not necessary hyperk\"{a}hler.}\]

(ii) Consider a Lagrangian fibration $f\colon X\to B$ and
let $\beta$ be the pull--back of an ample class $\alpha\in H^2(B,\QQ)$. Up to renumbering, the weight filtration associated with the class $\beta$ on $H^*$ centered at $n$
coincides with the perverse filtration, see Section \ref{sec:perversefiltra}
\[W^{\beta}_k H^d(X, \QQ)=P_{d+k-2n}H^{d}(X, \QQ).\]
Indeed, the action of $\beta$ gives the morphisms
\[
\beta\colon P_kH^d(X, \QQ) \to P_kH^{d+2}(X, \QQ) \qquad
\beta^{j}\colon \Gr^P_{i}H^{n+i-j} \simeq \Gr^P_{i}H^{n+i+j}.
\]
The isomorphism is called the \emph{perverse Hard Lefschetz theorem} \cite[Prop.\ 5.2.3]{deCataldoMigliorini2005}. By Proposition \ref{prop:C[x]/x^n+1}, this corresponds to the isomorphism $\bar{\sigma}^j\colon H^{n-j}(X, \Omega_X^i) \simeq H^{n+j}(X, \Omega_X^i)$.\smallskip

(iii) Let $\pi \colon \mathcal{X} \to \Delta$ be a projective degeneration of hyperk\"{a}hler manifolds over the
unit disk which we assume to be semistable, i.e.\ the central fiber $\mathcal{X}_0$ is reduced with simple normal crossings. For $t \in \Delta^*$, let $N$ denote the logarithmic monodromy operator on $H^{*}(\mathcal{X}_t, \QQ)$. 
The weight filtration of $N$ centered at $d$ on $H^d(\mathcal{X}_t, \QQ)$, denoted by $W_{k}H^d(\mathcal{X}_t, \QQ)$, is the weight filtration of the limit mixed Hodge structure associated to $\pi$, see \cite[Thm.\ 11.40]{PetersSteenbrink2008}. 

The
degeneration $\pi \colon \mathcal{X} \to \Delta$ is called of type III if $N^{2} \neq 0$ and $N^{3}=0$ on $H^2(\mathcal{X}_t, \QQ)$. In this case, the limit mixed Hodge structure is of Hodge--Tate type by \cite[Thm.\ 3.8]{Soldatenkov2020}, and in particular $\Gr^W_{2i+1} H^{*}(\mathcal{X}_t, \QQ)=0$. Then the even graded pieces of the weight filtration are used to define the \emph{Hodge numbers}
\[{^{\mathfrak{w}}}h^{i,j}(\mathcal{X})\coloneqq \dim \Gr^W_{2i} H^{i+j}(\mathcal{X}_t, \QQ).\]

 The Hodge numbers ${^{\mathfrak{w}}}h^{0,j}(\mathcal{X})$ have a clear geometric description. The dual complex of $\mathcal{X}_0 = \sum \Delta_i$, denoted by $D(\mathcal{X}_0)$, is the CW complex whose $k$-cells are in correspondence with the irreducible components of the intersection of $(k+1)$ divisors $\Delta_i$. The Clemens--Schmid exact sequence then gives 
\begin{equation}\label{eq:dualcomplex}
   {^{\mathfrak{w}}}h^{0,j}(\mathcal{X})=\dim H^{j}(D(\mathcal{X}_0), \QQ),
\end{equation}
see for instance \cite[\S 3, Cor.\ 1 \& 2]{Morrison1984}.

In order to show P$=$W, namely that the filtrations (ii) and (iii) can be identified, we need the notion of hyperk\"{a}hler triples with their associated $\mathfrak{so}(5, \CC)$-action.
\subsection{Hyperk\"{a}hler triples} 
A hyperk\"{a}hler manifold is a Riemannian manifold $(X, g)$ which is K\"ahler with respect to three complex structures $I$, $J$,
and $K$, satisfying the standard quaternion relations $I^2=J^2=K^2=IJK=-\mathrm{Id}$. The corresponding hyperk\"{a}hler triple is the triple of K\"{a}hler classes in $H^2(X, \CC)\times  H^2(X, \CC)\times  H^2(X, \CC) $ given by
\[(\omega_I, \omega_J, \omega_K)\coloneqq (g(I \cdot, \cdot), g(J \cdot, \cdot), g(K \cdot, \cdot)).\]
The set of all hyperk\"{a}hler triples on $X$ forms a Zariski-dense subset in 
\[D^{\circ}=\{(x,y,z)\mid \, q(x)=q(y)=q(z)\neq 0, q(x,y)=q(y,z)=q(z,x)=0\}.\]
In particular, all algebraic relations that can be formulated for triples in $D^{\circ}$ 
and which hold for triples of the form $(\omega_I , \omega_J, \omega_K)$ hold in fact for all $(x, y, z) \in D^{\circ}$, see \cite[Prop.\ 2.3]{ShYi18:ShenYin}. 

\subsection{The $\mathfrak{so}(5, \CC)$-action}\label{sec:so5action}
Recall the scaling operator \[H\colon H^i(X, \CC) \to H^i(X, \CC) \qquad H(v)=(i-2n)v.\]
By the Jacobson--Morozov theorem, to any $\omega \in H^2(X, \CC)$ of Lefschetz type we can associate a $\mathfrak{sl}_2$-triple $(L_{\omega}, H, \Lambda_{\omega})$. Let $p=(x,y,z) \in D^{\circ}$. 
The $\mathfrak{sl}_2$-triples associated to $x$, $y$ and $z$ generate the Lie subalgebra $\mathfrak{g}_{p} \subset \mathrm{End}(H^*(X, \CC))$, isomorphic to
$
\mathfrak{so}(5, \CC)$,
with Cartan subalgebra
\begin{equation}\label{eqn:HHp}
\mathfrak{h}=\langle H, H'_{p} \coloneqq \sqrt{-1}[L_y,\Lambda_z]\rangle.
\end{equation}
There is an associated weight decomposition
\begin{equation}\label{eq0}
    H^*(X, \CC)=\bigoplus_{i,j}H^{i,j}(p)
\end{equation}
such that for all $v\in H^{i,j}(p)$ we have
\[
H(v)=(i+j-2n)v \qquad H'_p(v)=(j-i)v.
\]
The following $\mathfrak{sl}_2$-triples in $\mathfrak{g}_p$
\begin{equation}\label{eq1}
    E_p\coloneqq \frac{1}{2}(L_y-\sqrt{-1}L_z) \qquad F_p\coloneqq \frac{1}{2}(\Lambda_y+\sqrt{-1}\Lambda_z) \qquad H_{p}\coloneqq\frac{1}{2}(H+H'_p),
\end{equation}
\begin{equation}\label{eq2}
   E'_p \coloneqq [E_p, \Lambda_x] \qquad\quad F'_p\coloneqq [L_x, F_p]
   \qquad\quad H'_p
\end{equation}
induce the same weight decomposition, since for any $v \in H^{i,j}(p)$ we have
\[
H_p(v)=(j-n)v \qquad H'_{p}(v)=(j-i)v.
\]

\begin{remark}\label{rmk:density}
The previous identities  for hyperk\"{a}hler triples are due to Verbitsky. The result for a general triple $p=(x,y,z) \in D^{\circ}$ follows from the density of hyperk\"{a}hler triples in $D^{\circ}$, and  the fact that the $\mathfrak{sl}_2$-representation $H^*(X, \CC)$ associated to $x$, $y$ and $z$ have the same weights, since $x$, $y$, and $z$ are all of Lefschetz type, see \cite[\S 2.4]{ShYi18:ShenYin}.
\end{remark}

\subsection{P$=$W}
The main result of \cite{HLSY19:ShenYin} is the following.
\begin{thm}[P$=$W]\label{P$=$W} 
For any Lagrangian fibration $f \colon X \to B$, there exists a type III projective
degeneration of hyperk\"{a}hler manifolds $\pi \colon \mathcal{X} \to \Delta$ with $\mathcal{X}_t$ deformation equivalent to $X$ for all $t \in \Delta^*$, together with a multiplicative isomorphism $H^*(X, \QQ) \simeq H^*(\mathcal{X}_t, \QQ)$, such that
\[
P_kH^*(X, \QQ)=W_{2k}H^*(\mathcal{X}_t, \QQ)=W_{2k+1}H^*(\mathcal{X}_t, \QQ).
\]
\end{thm}

\begin{proof}
Let $\beta = f^*\alpha$ be the pullback of an ample class $\alpha \in H^2(B, \QQ)$, and $\eta \in H^2(X, \QQ)$ with $q(\eta)>0$.
Since $\beta^{n+1}=0$, we have $q(\beta)=0$. Up to replacing $\eta$ with $\eta+ \lambda \beta$ for some $\lambda \in \QQ$, we can suppose that $q(\eta)=0$. Set
\[y=\beta + \eta \qquad z=-\sqrt{-1}(\eta-\beta). \]
By scaling a nonzero vector $x \in H^2(X,\CC)$ perpendicular to $y$ and $z$ with
respect to $q$, we obtain $p({f}) = (x,y,z)\in D^{\circ}$ with 
\[ \beta = \frac{1}{2}(y - \sqrt{-1}z).\]

Soldatenkov showed that the nilpotent operator $E'_{p(f)}$ is the logarithmic monodromy $N$
 of a projective type III  degeneration $\pi\colon\mathcal{X} \to \Delta$ of compact hyperk\"{a}hler manifolds deformation equivalent to $X$, see \cite[Lem.\ 4.1, Thm.\ 4.6]{Soldatenkov2020}.
 \footnote{One can use the Lie algebra structure of the LLV algebra to compare the present description of $E'_{p(f)}$ with that of \cite[Lem.\ 4.1]{Soldatenkov2020}, see \ \cite[Lem.\ 3.9]{KSV2019}. Mind that Soldatenkov's existence result is not constructive: it relies on lattice theory and the geometry of the period domain, and does not produce an explicit type III degeneration.} 

The weight decomposition for the $\mathfrak{sl}_2$-triple \eqref{eq1} splits the perverse filtration associated to $f$,  since $E_{p(f)}$ acts in cohomology via the cup product by $\beta$. The weight decomposition for the $\mathfrak{sl}_2$-triple \eqref{eq2} splits the weight filtration of the limit mixed Hodge structure associated to $\pi$, because $E'_{p(f)}=N$. Hence, by Section \ref{sec:so5action}, 
this implies P$=$W.
\end{proof}

P$=$W also provides alternative proofs of Corollary \ref{cor:ShenYin} and Theorem \ref{prop:CohBFibre}.

\begin{cor}[Numerical P$=$W] \label{cor:numericalP=W} ${^\mathfrak{p}}h^{i,j}(X)={^\mathfrak{w}}h^{i,j}(\mathcal{X}) 
=h^{i,j}(X)$.
\end{cor}
\begin{proof} 
By Theorem \ref{P$=$W} we obtain 
${^{\mathfrak{p}}}h^{i,j}(X)={^{\mathfrak{w}}}h^{i,j}(\mathcal{X})$. The equality 
${^{\mathfrak{p}}}h^{i,j}(X)=h^{i,j}(X)$ is Corollary \ref{cor:ShenYin}. 

Alternatively, one can argue as follows. By \cite[Thm.\ 3.8]{Soldatenkov2020}, the limit mixed Hodge structure $(H^{*}_{\mathrm{lim}}(\mathcal{X}_t, \QQ)\simeq H^*(\mathcal{X}_t, \CC), W_*, F_*)$ associated to $\pi$ is of Hodge--Tate type, and so
$
    {^{\mathfrak{w}}}h^{i,j}(\mathcal{X})=\dim_{\CC} \Gr^F_i H^{i+j}_{\mathrm{lim}}(\mathcal{X}_t, \CC).
$
By the classical result \cite[Cor.\ 11.25]{PetersSteenbrink2008}, we have 
$
    \dim_{\CC} \Gr^F_i H^{i+j}_{\mathrm{lim}}(\mathcal{X}_t, \CC)=h^{i,j}(\mathcal{X}_t).
$
We conclude that ${^{\mathfrak{p}}}h^{i,j}(X)=h^{i,j}(\mathcal{X}_t)=h^{i,j}(X)$.
\end{proof}

\begin{cor} At the boundary of the Hodge diamond of $X$, P$=$W gives\footnote{The identity $\dim H^{j}(D(\mathcal{X}_0), \QQ)=\dim H^j(\PP^n)$ was first proved in \cite[Thm.\ 7.13]{KLSV2018}.}
\begin{align*} 
\dim H^j(B, \QQ) & ={^{\mathfrak{p}}}h^{0,j}(X)  =h^{0,j}(X)=\dim H^j(\PP^n, \QQ),\\
\dim H^{j}(D(\mathcal{X}_0), \QQ)  & ={^{\mathfrak{w}}}h^{0,j}(\mathcal{X}) =h^{0,j}(X)=\dim H^j(\PP^n, \QQ),\\
\dim\mathrm{Im}(H^i(X,\QQ) \to H^i(X_t,\QQ)) & ={^{\mathfrak{p}}}h^{i,0}(X)=h^{i,0}(X)=\dim H^i(\PP^n, \QQ).
\end{align*}
\end{cor}
In the following, we provide conjectural conceptual explanations for these identities.

\subsection{A conjectural explanation I}
 Assume that $\mathcal{X}$ is Calabi--Yau. This can be always achieved via a MMP, at the cost of making $\mathcal{X}_0$ mildly singular (precisely divisorial log terminal), see \cite{Fujino2011}. 
  Under this assumption the homeomorphism class of $D(\mathcal{X}_0)$ is well-defined. 

Then the SYZ conjecture predicts that $\mathcal{X}_t$ carries a special Lagrangian fibration $f\colon \mathcal{X}_t \to D(\mathcal{X}_0)$ with respect to a hyperk\"{a}hler metric. By hyperk\"{a}hler rotation \cite[\S 3]{Hitchin2000}, 
$f$ should become a holomorphic Lagrangian fibration $f\colon X\to B$ on a hyperk\"{a}hler manifold $X$ deformation equivalent to $\mathcal{X}_t$. It is conjectured that the base of a Lagrangian fibration on $X$ is a projective space.
So in brief, we should have the homeomorphisms
\begin{equation}\label{conjSYZ}
    D(\mathcal{X}_0) \simeq \PP^{n} \simeq B.
\end{equation}
The latter equality is known to hold if $n\leq 2$ , see \S \ref{sec:sing}, or conditional to the smoothness of the base \cite{Hwang:fibrationsbase}. The former equality is known for degenerations of Hilbert schemes or generalised Kummer varieties \cite{BrownMazzon}. In both case, the most delicate problem is to assess the smoothness of $D(\mathcal{X}_0)$ or $B$. From this viewpoint, the identity
\[
\dim H^{j}(D(\mathcal{X}_0), \QQ)=\dim H^j(\PP^n, \QQ)= \dim I\!H^j(B, \QQ)=\dim H^j(B, \QQ).
\]
is a weak cohomological evidence for the conjecture \eqref{conjSYZ}.

\subsection{A conjectural explanation II}
We conjecture that the equality ${^{\mathfrak{p}}}h^{i,0}(\mathcal{X})={^{\mathfrak{w}}}h^{i,0}(\mathcal{X})$ is the result of the identification of two Lagrangian tori up to isotopy.
\begin{definition}
Let $x$ be a zero-dimensional stratum of $\mathcal{X}_0$. Choose local coordinates $z_0, \ldots, z_{2n}$ centered at $x$ with $\pi(z)=z_0 \cdot \ldots \cdot z_{2n}$. For fixed radii $0 < r_i \ll 1$ and $t=\prod^{2n}_{i=0}r_i$, a \emph{profound torus} $\mathbb{T} \subset \mathcal{X}_t$ is 
\[\mathbb{T} = \{ (r_0 e^{i \theta_0}, \ldots, r_{2n} e^{i \theta_{2n}})\mid \theta_0, \ldots, \theta_{2n} \in [0,2\pi), \, \theta_0+ \cdots+ \theta_{2n} - \arg(t)\in \mathbb{Z}\}.\]
\end{definition}

\begin{remark}
The ambient-isotopy type of $\mathbb{T} \subset \mathcal{X}_t$ does not depend on the choice of the coordinates: 
$\mathbb{T}$ is homotopic to $U_x \cap \mathcal{X}_t$, where $U_x$ is a neighbouhood of $x$ in $\mathcal{X}$. More remarkably, if $\mathcal{X}$ is Calabi--Yau, then the isotopy class of $\mathbb{T}$ in $\mathcal{X}_t$ is independent of $x$. This follows at once from Koll\'{a}r's notion of $\PP^1$-link (see \cite[Prop.\ 4.37]{Kollar2013} 
or \cite[Lem.\ 3.10]{Harder2019}
), or equivalently because profound tori are fibers of the same smooth fibration, by adapting \cite[Prop.\ 6.12.]{EvansMauri2019} 
\end{remark}

\begin{conj}[Geometric P$=$W]\label{GeometricP$=$W}
For any Lagrangian fibration $f \colon X \to B$ with general fiber $T$, there exists a projective minimal dlt type III 
degeneration of hyperk\"{a}hler manifolds $\pi \colon \mathcal{X} \to \Delta$ with $\mathcal{X}_t$ deformation equivalent to $X$ for all $t \in \Delta^*$, such that $T$ is isotopic to a profound torus $\mathbb{T}$.
\end{conj}

The conjecture is inspired by the geometric P$=$W conjecture for character varieties, see the new version of \cite{MMS} (to appear soon). 
Lemma \ref{lem:Voisinrest} and \eqref{eq:P0Q*} give
\[P_{d-1}H^d(X,\QQ)=\mathrm{Ker}\left(H^d(X, \QQ)\rightarrow H^d(T, \QQ)\right).\]
If $\mathcal{X}_0$ has simple normal crossings (or dlt singularities modulo adapting \cite[Thm.\ 3.12]{Harder2019}), one obtains that 
\[W_{2d-1}H^d(\mathcal{X}_t, \QQ)=\mathrm{Ker}\left(H^d(\mathcal{X}_t, \QQ)\rightarrow H^d(\mathbb{T}, \QQ)\right).\]
Therefore, Conjecture \ref{GeometricP$=$W} would give a geometric explanation of P$=$W at the highest weight
\[P_{d-1}H^d(X, \QQ) = W_{2d-1}H^d(\mathcal{X}_t, \QQ).\]
It is not clear what a geometric formulation of P$=$W should be that could explain the cohomological statement in all weights.

Recent advance in the SYZ conjecture due to Yang Li \cite{YangLi} 
suggests that
profound tori can be made special Lagrangian, modulo a conjecture in non-archimedean geometry. A few months ago, the existence of a single special Lagrangian torus on $\mathcal{X}_t$ was a complete mystery, see \cite[\S 5, p.152]{Gross2013}. 
Note also that Li's result is compatible with the expectation in symplectic geometry \cite[Conj.\ 7.3]{Auroux2007}. 
  Profound tori appear as general fibers of the SYZ fibration that Li constructed on an open set which contains an arbitrary large portion of the mass of $\mathcal{X}_t$ with respect to a Calabi--Yau metric, still modulo the non-archimedean conjecture. It is curious (but maybe not surprising) that also the previously quoted results \cite{HuXu21:HuybrechtsXu} and \cite{BrownMazzon} highly rely on non-archimedean techniques. 
  
\subsection{Multiplicativity of the perverse filtration}  P$=$W implies that the perverse filtration on $H^*(X, \QQ)$ is compatible with cup product.

\begin{cor}[Multiplicativity of the perverse filtration] 
Assume $f\colon X \to B$ is a fibration. Then the perverse filtration on $H^*(X, \QQ)$ is multiplicative under cup product, i.e.\
\[
\cup\colon P_k H^d(X,\QQ) \times P_{k'}H^{d'}(X, \QQ) \to P_{k+k'} H^{d+d'}(X,\QQ).
\]
\end{cor} 
\begin{proof}
By P=W, it is sufficient to show that the weight filtration is multiplicative. To this end, endow the tensor product $H^*(\mathcal{X}_t, \QQ)\otimes H^*(\mathcal{X}_t, \QQ)$ with the nilpotent endomorphism $N^{\otimes}\coloneqq N \otimes 1 + 1 \otimes N$, and call $W^{\otimes}$ the weight filtration of $N^{\otimes}$. Since the monodromy operator $e^{N}$ is an algebra homomorphism of $H^*(\mathcal{X}_t, \QQ)$, $N$ is a derivation, i.e.\
\[ N(x \cup y)= Nx\cup y + x \cup Ny= \cup(N^{\otimes}(x \otimes y)).\]
As a consequence, the construction of the weight filtration (see \ Section \ref{defn:weightfiltration}) gives
\[
\cup(W^{\otimes}_k (H^i(\mathcal{X}_t, \QQ)\otimes H^j(\mathcal{X}_t, \QQ)))\subseteq W_kH^{i+j}(\mathcal{X}_t, \QQ).\]
Together with \cite[1.6.9.(i)]{Deligne80} which says that \[W^{\otimes}_k (H^i(\mathcal{X}_t, \QQ)\otimes H^j(\mathcal{X}_t, \QQ))= \bigoplus_{a+b=k} W_a H^i(\mathcal{X}_t, \QQ)\otimes W_b H^j(\mathcal{X}_t, \QQ),\] we conclude that the weight filtration is multiplicative. Alternatively see \cite[\S 5]{HLSY19:ShenYin}.
\end{proof}
\begin{remark}
For an arbitrary morphism of projective varieties or K\"{a}hler manifolds, the perverse filtration is not always multiplicative \cite[Exa.\ 1.5]{Zhang2017}, but it is so for instance if it coincides with the Leray filtration, or if P$=$W holds. Indeed, the Leray filtration and the weight filtration of the limit mixed Hodge structure are multiplicative.

It is natural to ask whether the multiplicativity holds at a sheaf theoretic level, for $Rf_*\QQ_X$, or over an affine base. The motivation for this comes from the celebrated P$=$W conjecture for twisted character varieties \cite{deCataldoHauselMigliorini2012}, which has been proved to be equivalent to the conjectural multiplicativity of the perverse filtration of the Hitchin map that is a proper holomorphic Lagrangian fibration over an affine base, see \cite[Thm.\ 0.6]{deCataldoMaulikShen2019}. From this viewpoint, it is remarkable that Shen and Yin give a proof of the multiplicativity in the compact case \cite[Thm.\ A.1]{ShYi18:ShenYin} which uses only the representation theory of $\mathfrak{sl}(2)$-triples, with no reference to the weight filtration.
\end{remark} 
  
\subsection{Nagai's conjecture for type III degenerations} Let $\pi\colon\mathcal{X} \to \Delta$ be a projective
degeneration of hyperk\"{a}hler manifolds with unipotent monodromy $T_d$ on $H^d(\mathcal{X}_t, \QQ)$. The \emph{index of nilpotence} of $N_d\coloneqq \log T_d$ is
\[\nilp(N_d)=\max\{i \mid N^i_d\neq 0\},\]
 and $\nilp(N_d) \leq d$ by \cite[Ch.\ IV]{Griffiths}.
 It is known that $H^2(\mathcal{X}_t, \QQ)$ determines the Hodge structure of $H^d(\mathcal{X}_t, \QQ)$ by means of the LLV representation, see \cite{SoldatenkovHodgeHK}. Nagai's conjecture investigates to what extent $\nilp(N_2)$ determines $\nilp(N_d)$.  The ring structure of the subalgebra generated by $H^2$ implies the inequality $\nilp(N_{2k})\geq k \cdot \nilp(N_{2})$, see \cite[Lem.\ 2.4]{Nagai}, but equality is expected.

\begin{conj}[Nagai]
$\nilp(N_{2k})=k \cdot \nilp(N_{2})$ for $k \leq 2n$.
\end{conj} 

The previous inequalities imply Nagai's conjecture for type III degenerations, i.e.\ $\nilp(N_{2})=2$. Remarkably, P$=$W explains Nagai's conjecture
in terms of the level of the Hodge structure $H^d(\mathcal{X}_t, \QQ)$, and determines $\nilp(N_d)$ even for $d$ odd. Recall that the \emph{level} of a Hodge structure $H = \oplus H^{p,q}$, denoted by $\level(H)$, is the largest difference $|p-q|$ for which $H^{p,q}\neq 0$, or equivalently the length of the Hodge filtration on $H$.

\begin{prop}
Let $\pi\colon\mathcal{X} \to \Delta$ be a type III projective degeneration of hyperk\"{a}hler manifolds with unipotent monodromy. 
Then 
\[\nilp(N_d)=\level(H^d(\mathcal{X}_t, \CC)).\]
For $k\leq 2n$, the following identities hold:
\begin{enumerate}
\item[(i)]$\nilp(N_{2k}) =2k= k \cdot \nilp(N_{2})$,
\item[(ii)] $\nilp(N_{2k+1}) =2k-1$, if $H^3(\mathcal{X}_t, \CC)\neq 0$.
\end{enumerate}
\end{prop}

\begin{remark}
The Statement (ii) is proved in \cite[Prop.\ 3.15]{Soldatenkov2020}. Here we present an alternative simple proof of (ii) which avoids the LLV representation. 

Nagai's conjecture is known to hold for degenerations of type I and III, i.e.\  for $\nilp(N_{2})=0$ and $2$, see \cite[Thm.\ 6.5]{KLSV2018}.  In order to establish Nagai's conjecture in full, only the case of type II degenerations
remains open, i.e.\ when $\nilp(N_{2})=1$. For type II there are partial results: $k \leq \nilp(N_{2k}) \leq 2k-2$ for $2 \leq k \leq n-1$, see \cite[Thm.\ 6.5]{KLSV2018}, and $\nilp(N_{2n})=n$, see \cite[Thm.\ 1.2]{HuybrechtsMauri21}. The full conjecture holds for all the known deformation types of hyperk\"{a}hler manifolds by \cite[Thm.\ 1.13]{GKLR2019}. Further comments on Nagai's conjecture for type II can be found in \cite{GKLR2019, Harder2020, HuybrechtsMauri21}.
\end{remark}
\begin{proof}
Let $l_d$ be half of the length of the weight filtration of $N_d$, i.e.\
$
l_d\coloneqq \min \{i \colon W_{2i}H^d(\mathcal{X}_t, \QQ) = H^d(\mathcal{X}_t, \QQ)\}.
$
By Definition \ref{defn:weightfiltration}, we have $\nilp(N_d)=l_d$.

For any type III degeneration of Hodge structures of hyperk\"{a}hler type with unipotent mono\-dromy, we know by the proof of Theorem \ref{P$=$W} that the logarithmic monodromy $N_*$ is of the form $E'_{p}=[\beta, \Lambda_x]$ for some $\beta$ and $x$ in $ H^2(X, \QQ)$ with $q(\beta)=0$. Here, we use the assumption $b_2(\mathcal{X}_t)\geq 5$, see \cite[\S 4.1]{Soldatenkov2020}. Then, by Corollaries \ref{cor:C[x]/x^n+1} and \ref{cor:numericalP=W}, we have $l_d=\level(H^d(\mathcal{X}_t, \CC))$. Hence, $\nilp(N_d)=\level(H^d(\mathcal{X}_t, \CC))$.

Finally, statements (i) and (ii) are equivalent to (i) $H^{2k,0}(\mathcal{X}_t)=\CC \sigma \neq 0$, and (ii) $H^{2k,1}(\mathcal{X}_t) \neq 0$ if $H^{2,1}(\mathcal{X}_t) \neq 0$, which follows from \eqref{eq:perverseHodgeHL}.
\end{proof}

\section{Examples and counterexamples}\label{sec:ex}
\begin{ex}\label{Ex1}
 In \cite[Ex.\ 1.7.(iv)]{Namikawa:deformation} Namikawa exhibits an example of a submanifold $T$ of a hyperk\"{a}hler manifold $X$ which is isomorphic to a complex torus, but is not Lagrangian (actually it is symplectic). 
 
 Let $E$, $F$ be elliptic curves defined by the cubic equations $f$ and $g$ respectively, and let $Y \subseteq \PP^5$ be the cubic fourfold given by the equation $h\coloneqq f(x_0,x_1,x_2)+g(y_0,y_1,y_2)=0$. The cyclic group $G \coloneqq \ZZ/3\ZZ$ acts on $Y$ by
 \[\phi_{\zeta}\colon [x_0:x_1:x_2:y_0:y_1:y_2]\mapsto [x_0:x_1:x_2:\zeta y_0:\zeta y_1: \zeta y_2],\]
 where $\zeta$ is a primitive third root of unity. The induced action on the Fano variety of lines $X$ is symplectic, i.e.\ $\phi_{\zeta}^*\sigma = \sigma$ for $\sigma \in H^0(X, \Omega^2_X)$. Indeed, by \cite{BD:fano_lines} there is a $G$-equivariant isomorphism $H^0(X, \Omega^2_X)\simeq H^1(Y,\Omega^3_Y)$. Denoting by $\Omega$ the canonical section of $H^0(\PP^5, K_{\PP^5}(6))$, $H^1(Y,\Omega^3_Y)$ is generated by the $G$-invariant residue $\mathrm{Res}_Y(\Omega/h^2)$, and so the action is symplectic. In particular, the fixed locus $T$ of the $G$-action on $X$ is a symplectic submanifold. One defines $T$ as the set of lines which join two points on $Y \cap \{y_0=y_1=y_2=0\} \simeq E$ and $Y \cap \{x_0=x_1=x_2=0\} \simeq F$ respectively. Hence, $T \simeq E \times F$. We conclude that $T$ is a symplectic torus embedded in the hyperk\"{a}hler manifold $X$.
\end{ex}

\begin{ex}
There exists a Lagrangian submanifold $L$ of a hyperk\"{a}hler manifold $X$ with \[\mathrm{Im}(H^2(X,\QQ) \to H^2(L, \QQ)) \not\simeq \QQ.\]
\end{ex}
\begin{proof}
Let $f\colon S \to \PP^1$ be an elliptic K3 surface with smooth fiber $E$. Define $L \subseteq X \coloneqq S^{[2]}$ to be the locus of non-reduced length-two subschemes of $S$ supported on $E$, which is isomorphic to the $\PP^1$-bundle $\PP(\Omega^1_{S}|_E)$ over $E$. Then, $L$ is an irreducible component of the fiber of the Lagrangian fibration $f^{[2]}\colon S^{[2]}\to S^{(2)} \to \PP^2$, thus $L$ is Lagrangian. 
The exceptional divisor $\rm Exc$ of the Hilbert--Chow morphism $S^{[2]}\to S^{(2)}$ restricts to a multiple of the tautological line bundle $\mathcal{O}_{\PP(\Omega^1_{S}|_E)}(-1)$ on $L$. Therefore, the second cohomology group $H^2(L)$ is generated by the restriction of  $\rm Exc$ and the pullback of an ample line bundle of $S^{(2)}$.
\end{proof}

\begin{ex}
There exists a Lagrangian submanifold $L$ of a hyperk\"{a}hler manifold $X$ with \[\mathrm{Im}(H^2(X,\QQ) \to H^2(L, \QQ)) \simeq \QQ \quad \text{ and }\quad  \mathrm{Im}(H^*(X,\QQ) \to H^*(L, \QQ)) \not\simeq H^*(\PP^{n}, \QQ).\]
\end{ex}
\begin{proof}
Let $C$ be a smooth curve of genus two in an abelian surface $A$. Consider the moduli space  $M_{\text{odd}}(A)$ of stable 1-dimensional sheaves on $A$ supported on the curve class
\[2[C] \in H_2(A, \ZZ)\]
and Euler characteristic $-1$. The fiber of the Albanese morphism $M_{\text{odd}}(A) \to A \times \widehat{A}$ is a compact hyperk\"{a}hler manifold $X$ deformation equivalent to a generalised Kummer variety of dimension six. Taking
Fitting supports defines a Lagrangian fibration 
\[X \to \PP^3 = |2C|.\]
The fiber over the curve $2C$ contains the locus $L$ of stable sheaves $\mathcal{F}$ on $A$ such that the composition $\mathcal{O}_S \to \mathcal{O}_{2C} \to \mathcal{E}nd_S(\mathcal{F})$ factors via the natural map $\mathcal{O}_{2C} \to \mathcal{O}_{C}$. As $\mathcal{O}_{C}$-module, $\mathcal{F}$ is a rank-two vector bundle, and  $L$ can be identified with the moduli space of rank-two vector bundles on $C$ of degree one, which is isomorphic to the intersection of two quadrics in $\PP^5$, see \cite{deCRS2021} and \cite{NarasimhanRamanan1969a}. The cohomology $H^*(X)$ is generated by so-called tautological classes, and $H^*(L)$ is generated by their restrictions, see \cite{Markman2002} 
 and \cite[Thm.\ 1]{Newstead1972}. 
Therefore, we have
\[
H^*(X, \QQ) \twoheadrightarrow H^*(L, \QQ) \simeq H^*(\PP^3, \QQ)\oplus \QQ^{4}[-3] \not \simeq H^*(\PP^3, \QQ). 
\]
\end{proof}

\textbf{Acknowledgement:} We wish to thank Paolo Cascini for bringing to our attention Koll\'{a}r's conjecture \ref{Kconj}, Thorsten Beckmann for useful conversations and in particular for suggesting Example \ref{Ex1}, Fabrizio Anella and Olivier Debarre for reading a first version of this note. The second author is supported by the
Max Planck Institute for Mathematics. 

\bibliography{HyperHM}
\bibliographystyle{alphaspecial}
\bibliographystyle{plain}
\end{document}